\documentclass[12pt]{amsart}
\usepackage{amssymb, amscd, amsmath, amsthm, latexsym, enumerate}
\newtheorem{theorem}{Theorem}
\newtheorem{lemma}[theorem]{Lemma}

\newtheorem*{cor}{Corollary}

\begin{document}

\title{$S^2$-bundles over 2-orbifolds}

\author{Jonathan A. Hillman }
\address{School of Mathematics and Statistics\\
     University of Sydney, NSW 2006\\
      Australia }

\email{jonh@maths.usyd.edu.au}

\begin{abstract}
Let $M$ be a closed 4-manifold with $\pi_2(M)\cong{Z}$.
Then $M$ is homotopy equivalent to either $CP^2$,
or the total space of an orbifold bundle 
with general fibre $S^2$ over an aspherical 2-orbifold $B$,
or the total space of an $RP^2$-bundle over an aspherical surface.
If $\pi=\pi_1(M)\not=1$ there are at most two such bundle spaces 
with given action $u:\pi\to{Aut}(\pi_2(M))$.
The bundle space has the geometry $\mathbb{S}^2\times\mathbb{E}^2$
(if $\chi(M)=0$) or $\mathbb{S}^2\times\mathbb{H}^2$ (if $\chi(M)<0$)
except when $B$ is orientable and $\pi$ is generated by involutions,
in which case the action is unique and there 
is one non-geometric orbifold bundle.
\end{abstract}



\maketitle

Every closed 4-manifold with geometry $\mathbb{S}^2\times\mathbb{E}^2$ 
or $\mathbb{S}^2\times\mathbb{H}^2$ has a foliation 
with regular leaves $S^2$ or $RP^2$.
The leaf space of such a foliation may be regarded as a compact 2-orbifold,
and the projection to the leaf space is an orbifold bundle projection.
If the regular leaves are $S^2$ the singularities of this orbifold
are cone points of order 2 or reflector curves.
If the regular leaves are $RP^2$ there are no exceptional leaves,
and the projection is an $RP^2$-bundle over a surface.
The total space of an $S^2$-or $RP^2$-bundle over an aspherical surface
is geometric, by Theorems 10.8 and 10.9 of \cite{Hi}.
In this paper we shall show that every closed 4-manifold $M$ 
with universal cover $\widetilde{M}\simeq{S^2}$ is homotopy equivalent 
to a manifold admitting an orbifold bundle structure, 
and in ``most" cases the bundle space is geometric.
Our main concern here is with the case of $S^2$-orbifold bundle spaces,
as the cases with no exceptional leaves are well understood.

Each pair $(\pi,u)$ where $\pi=\pi^{orb}(B)$ is a 2-orbifold group
and $u:\pi\to{Z/2Z}$ is an epimorphism with torsion-free kernel is
realized by an $\mathbb{S}^2\times\mathbb{E}^2$- or
$\mathbb{S}^2\times\mathbb{H}^2$-manifold $M$.
In \S1 we give a ``standard" example $M_{st}$, 
and review some of its algebraic invariants. 
In \S2 we consider local models for orbifold bundle projections,
and in \S3 we show that there are at most two 4-manifolds $M$ which 
are total spaces of orbifold bundles over $B$ with regular fibre $S^2$
and action $u$ on $\pi_2(M)\cong{Z}$.
(The two manifolds differ at most by ``Gluck reconstruction" of
a product neighbourhood of a regular fibre.)
The base orbifold $B$ must have a nonsingular double cover.
In particular, its singular locus consists of cone points 
of order 2 and reflector curves.
If $B$ has a reflector curve, the bundle is unique.
We show also that if $B$ is an $\mathbb{H}^2$-orbifold then every such bundle
space is either geometric or has a decomposition into two geometric pieces.
In \S4 we review briefly the cases with spherical base orbifold.

We return to the homotopy classification in \S5,
where we show that if $M$ is any 4-manifold realizing $(\pi,u)$ 
then $k_1(M)=k_1(M_{st})$.
In \S6 we construct an explicit model for the second stage 
$P$ of the Postnikov tower for $M_{st}$,
and show that ``Gluck reconstruction" changes the
image of $[M]$ in $H_4(P;\mathbb{F}_2)$.
Our main result is Theorem 12, in \S7,
where we show that if $\pi\not=1$, $\pi_2(M)\cong{Z}$ and
$\pi\not\cong\mathrm{Ker}(u)\times{Z/2Z}$ 
then $M$ is homotopy equivalent to an $S^2$-orbifold bundle space.
(The cases with $\pi=1$ or $\pi\cong\mathrm{Ker}(u)\times{Z/2Z}$ 
were already known.)
Our argument derives ultimately from \cite{HK}.
If the base orbifold $B$ has reflector curves there is 
an unique homotopy type realizing the pair $(\pi,u)$,
and this is represented by a geometric 4-manifold.
If $B$ has only cone point singularities there are two homotopy types,
and if $\pi$ is not generated by involutions both
homotopy types are represented by geometric 4-manifolds.
However, if $B$ is the orbifold quotient of an orientable surface 
by the hyperelliptic involution only one of these is geometric.
The second Wu class $v_2(M)$ is an essential invariant for
$S^2$- and $RP^2$-bundles. 
However in \S8 we show that if $M$ is an $S^2$-orbifold bundle space 
and $B$ has singularities then $v_2(M)$ is determined by $\pi$.

In the final three sections we show first that the 22 
$\mathbb{S}^2\times\mathbb{E}^2$-manifolds have distinct homotopy types,
and there is one more homotopy type represented by a non-geometric
$S^2$-orbifold bundle over $S(2,2,2,2)$.
The TOP structure sets of such manifolds are infinite if 
$\pi$ has torsion but is not a product with $Z/2Z$.
If moreover $\pi/\pi'$ is finite then there are infinitely many
homeomorphism types within each such homotopy type.
Finally we apply the main result to a characterization of the
homotopy types  of orientable 4-manifolds which are total spaces of bundles 
over $RP^2$ with aspherical fibre and a section.

I would like to thank Wolfgang L\"uck, for computing the surgery obstruction
groups $L_*(\pi,w)$ for the $\mathbb{E}^2$-orbifold groups
(for all orientation characters) at my request \cite{Lu10},
and Elmar Vogt, for sending me a scanned copy of the final chapter 
of his dissertation \cite{Vo70},
in which he classified $S^2$-orbifold bundles over
orbifolds with no reflector curves.

\section{the standard example}

Although we shall consider quotients of $S^2\times{S^2}$ briefly in \S3,
our main concern is with 4-manifolds $M$ covered by $S^2\times{R}^2$.
We shall identify $S^2$ with $CP^1=\mathbb{C}\cup\{\infty\}$,
via stereographic projection from 
$(0,1)\in{S^2}\subset\mathbb{C}\times\mathbb{R}$.
Under this identification the antipodal map $a$ is given by
$a(z)=-z/|z|^2$ (i.e, $a([z_0:z_1])=[-\overline{z_1}:\overline{z_0}]$),
and rotation through an angle $\theta$ about the axis through $0$ and 
$\infty$ is given by $R_\theta(z)=e^{i\theta}z$.
(Care! Multiplication by $-1$ in $CP^1$ is $R_\pi$, 
not $a$!) 
We shall identify the groups $\mathbb{Z}^\times=\{\pm1\}$, 
$Z/2Z$ and $\mathbb{F}_2$, where appropriate.

Let $M$ be a closed 4-manifold with $\pi_2(M)\cong{Z}$
and $\pi=\pi_1(M)\not=1$, 
and let $u:\pi\to{Aut(\pi_2(M))}=\mathbb{Z}^\times$ be the natural action.
Let $U\in{H^1(\pi;\mathbb{F}_2)}=Hom(\pi,Z/2Z)$ be 
the cohomology class corresponding to the epimorphism $u$.
Then $M$ has universal cover $\widetilde{M}\cong{S^2}\times{R^2}$ and
$\kappa=\mathrm{Ker}(u)$ is a $PD_2$-group,
and $w=w_1(M)$ is determined by the pair $(\pi,u)$.
In particular, $w_1(M)|_\kappa=w_1(\kappa)$, 
since $\kappa$ acts trivially on $\pi_2(M)$.
(See Chapter 10 of \cite{Hi}.
Note that if $u$ is nontrivial $\pi$ may have automorphisms
that do not preserve $u$.)
Let $[M]\in{H_4(M;Z^w)}\cong{Z}$ be a fundamental class.

If $\pi$ is torsion-free then $M$ is TOP $s$-cobordant 
to the total space of an $S^2$-bundle over an aspherical surface.
If $\pi\cong\kappa\times{Z/2Z}$ then any 4-manifold $M$ with 
$\pi_1(M)\cong\pi$ and $\pi_2(M)\cong{Z^u}$ is simple homotopy equivalent 
to the total space of an $RP^2$-bundle over $K(\kappa,1)$.
For each $PD_2$-group $\kappa$ there are two such bundles, 
distinguished by whether $v_2(M)=0$ or not.
As these cases are well-understood, we shall
usually assume that $M$ is not homotopy equivalent to a bundle space.

If $\pi$ has torsion but is not a direct product
then $u$ is nontrivial and $\pi\cong\kappa\rtimes{Z/2Z}$.
Moreover $\pi$ is the orbifold fundamental group of a $\mathbb{E}^2$-
or $\mathbb{H}^2$-orbifold $B$.
Since $\kappa$ is torsion free the singular locus $\Sigma{B}$ 
consists of cone points of order 2 and reflector curves.

The surface $K(\kappa,1)$ has an involution $\zeta$ 
corresponding to the action of $\pi/\kappa\cong{Z/2Z}$.
The ``standard" example of a closed 4-manifold realizing $(\pi,u)$ is 
\[M_{st}=S^2\times{K(\kappa,1)}/(s,k)\sim(-s,\zeta(k)).\]
This is a $\mathbb{S}^2\times\mathbb{E}^2$-manifold if $\chi(\pi)=0$,
and is a $\mathbb{S}^2\times\mathbb{H}^2$-manifold otherwise.
Projection onto the first factor induces a bundle projection
from $M_{st}$ to $RP^2$, with fibre $F=K(\kappa,1)$.
In particular, $U^3=0$, since $U$ is induced from the generator of
$H^1(RP^2;\mathbb{F}_2)$.
Projection onto the second factor induces an orbifold 
bundle projection $p_{st}:M_{st}\to{B}$ with regular fibre $F\cong{S^2}$.

The algebraic 2-type $[\pi,\pi_2(M),k_1(M)]$ determines $P_2(M)$, 
the second stage of the Postnikov tower for $M$, 
and the homotopy type of $M$ is determined by the
image of $[M]$ in $H_4(P_2(M);Z^w)$,
modulo the action of $Aut(P_2(M))$.
There are at most two possible values for this image, 
up to sign and automorphisms of the algebraic 2-type,
by Theorem 10.6 of \cite{Hi}.
It is clear from this Theorem that the homotopy type of $M$ 
is in fact detected by the image of $[M]$ in $H_4(P;\mathbb{F}_2)$.
We shall construct a model for $P_2(M_{st})$ in \S6.

\section{local models for orbifold bundles}

A cone point of order 2 in a 2-orbifold has a regular neighbourhood
which is orbifold-homeomorphic to $D(2)=D^2/d\sim-d$.
Let $\mathbb{J}=[[0,1]=[-1,1]/x\sim-x$ be the compact connected 1-orbifold 
with one reflector point.
A reflector curve (with no corner points) in a 2-orbifold 
has a regular neighbourhood which is orbifold-homeomorphic 
to $\mathbb{J}\times{S^1}$.
However there are two possible surjections 
$u:{\pi^{orb}(\mathbb{J}\times{S^1})}\to{Z/2Z}$ with torsion-free kernel.
We shall say that the curve is {\it $u$-twisted}
if the cover is the M\"obius band $Mb={[-1,1]\times{S^1}/(x,u)\sim(-x,-u)}$ 
with the involution $[x,u]\mapsto[-x,u]=[x,-u]$;
if the cover is $[-1,1]\times{S^1}$ with involution $(x,u)\mapsto(-x,u)$
we shall say that the curve is untwisted.
(Note that this notion involves both the reflector curve and the action.)

For example, as the quotient of an involution of the torus $T$
the ``silvered annulus" $\mathbb{A}=S^1\times{S^1}/(u,v)\sim(u,\bar{v})$ 
has two untwisted reflector curves.
However it is also the quotient of an involution of the Klein bottle $Kb$,
and the reflector curves are then both twisted.
On the other hand, the ``silvered M\"obius band" 
$\mathbb{M}b={S^1\times{S^1}/(u,v)\sim(v,u)}$
has two distinct (but isomorphic) nonsingular covers, 
but in both cases the reflector curve is untwisted.

Models for regular neighbourhoods of the exceptional fibres
of such orbifold bundles may be constructed as follows.
Let 
\[E(2)=S^2\times{D^2}/(z,w)\sim(a(z),-w),\]
\[\mathbb{E}=S^2\times[-1,1]\times{S^1}/(z,x,u)\sim(a(z),-x,u)\]
and 
\[\mathbb{E}'=S^2\times[-1,1]\times{S^1}/(z,x,u)\sim(a(z),-x,u)\sim(z,-x,-u).\]
Then $p_2([z,w])=[w]$, $p_\mathbb{E}([z,x,u])=[u,x]$
and $p_{\mathbb{E}'}([z,x,u])=[x,u]$
define bundle projections $p_2:E(2)\to{D(2)}$,
$p_\mathbb{E}:\mathbb{E}\to{\mathbb{J}\times{S^1}}$
(with untwisted reflector curve) 
and $p_{\mathbb{E}'}:\mathbb{E}'\to{\mathbb{J}}\times{S^1}$
(with twisted reflector curve).
Any $S^2$-bundle over $\mathbb{J}\times{S^1}$ or $D(2)$ 
with nonsingular total space must be of this form.
The other local models for nontrivial actions on the fibre have base 
$Mb$ and total space $S^2\times{Mb}$ (non-orientable) or
$S^2\times[-1,1]\times[0,1]/(z,t,0)\sim(a(z),-t,1)$ (orientable).

It is also convenient to let
$D(2,2)=[-1,1]\times{S^1}/(x,u)\sim(-x,\bar{u})$
be the disc with two cone points of order 2 and
\[E(2,2)=S^2\times[-1,1]\times{S^1}/(z,x,u)\sim(a(z),-x,\bar{u}),\]
with projection $p_{2,2}([z,x,u])=[x,u]$.
Then $D(2,2)$ is the boundary-connected-sum of two copies of $D(2)$,
and $E(2,2)$ is the corresponding fibre sum of two copies of $E(2)$.

The manifolds $E(2)$ and $\mathbb{E}'$ have boundary $S^2\tilde\times{S^1}$, 
and $p_2|_{\partial{E(2)}}$ and $p|_{\partial\mathbb{E}'}$
are nontrivial $S^2$-bundles over $S^1$.
In all the other cases the restriction of the fibration over
the boundary of the base orbifold is trivial.
(When the base is $B=Mb$ or $D(2,2)$ this can be seen by
noting that $\partial{B}$ is homotopic to the product of two generators
of $\pi_1^{orb}(B)$, and considering the action on $\pi_2(E)\cong{Z}$.)
For later uses we may need to choose homeomorphisms
$\partial{E}\cong{S^2}\times{S^1}$.

Let $\alpha,\beta$ and $\tau$ be the self-homeomorphisms 
of $S^2\times{S^1}$ defined by $\alpha(z,u)=(a(z),u)$, 
$\beta(z,u)=(z,\bar{u})$ and
$\tau(z,u)=(uz,u)$, for all $(z,u)\in{S^2\times{S^1}}$.
The images of $\alpha,\beta$ and $\tau$ generate
$\pi_0(Homeo(S^2\times{S^1}))\cong(Z/2Z)^3$.
The group $\pi_0(Homeo(S^2\tilde\times{S^1}))\cong(Z/2Z)^2$
is generated by the involution $\tilde\beta([z,u])=[z,\bar{u}]$
and the twist $\xi([z,u])=[uz,u]$.

\begin{lemma}
\begin{enumerate}
\item The self-homeomorphisms $\alpha$ and $\beta$ of $S^1\times{S^2}$ 
extend to fibre-preserving self-homeomorphisms of $S^2\times{D^2}$ and $E(2,2)$.

\item Every self-homeomorphism of $S^1\times{S^2}$ extends 
to a fibre-preserving self-homeomorphism of $\mathbb{E}$.

\item The self-homeomorphism $\tilde\beta$ of $S^2\tilde\times{S^1}$
extends to fibre-preserving self-homeomorphisms of $E(2)$
and $\mathbb{E}'$.
\end{enumerate}
\end{lemma}

\begin{proof}
It is sufficient to check that the above representatives of 
the isotopy classes extend, which in each case is clear. 
\end{proof}

However $\tau$ does not extend across $S^2\times{D^2}$ or $E(2,2)$,
as we shall see.
Nor does $\xi$ extend across $E(2)$ or $\mathbb{E}'$.

\section{general results on orbifold bundles}

Let $M$ be a closed 4-manifold which is the total space of
an orbifold bundle $p:M\to{B}$ with regular fibre $F\cong{S^2}$ 
over the 2-orbifold $B$. 
Then $\pi_1^{orb}(B)\cong\pi_1(M)$.
Let $\Sigma{B}$ be the singular locus of $B$.
For brevity, we shall say that $M$ is an {\it $S^2$-orbifold bundle space}
and $p$ is an {\it $S^2$-orbifold bundle}.

\begin{lemma}
The singular locus $\Sigma{B}$ consists of cone points of
order $2$ and reflector curves (with no corner points).
The number of cone points plus the number of $u$-twisted reflector curves
is even.
In particular, the base orbifold must be good.
There is a cone point if and only if $\pi=\pi_1^{orb}(B)$
has an element $x$ of order $2$ such that $w(x)\not=0$,
and there is a reflector curve if and only if $\pi$ has
an element $x$ of order $2$ such that $w(x)=0$.
\end{lemma}

\begin{proof}
The first assertion holds since the stabilizer of a point 
in the base orbifold must act freely on the fibre $S^2$.

Let $N$ be a regular neighbourhood of $\Sigma{B}$,
and let $V$ be the restriction of $U$ to $B\setminus{N}$.
Then $V(\partial{N})=0$.
The action $u$ is trivial on boundary components of $N$ 
parallel to untwisted reflector curves, 
but is nontrivial on all other boundary components.
Therefore $V(\partial{N})$ is the sum of the number of cone points
and the number of $u$-twisted reflector curves, modulo $(2)$.
Thus this number must be even, and $B$ cannot be $S(2)$, 
which is the only bad orbifold in which all point stabilizers 
have order at most 2.

The final assertions follow since an involution of a surface
with a fixed point point is either locally a rotation
about an isolated fixed point or locally a reflection across a
fixed arc.
\end{proof}

If $B$ is spherical then $\widetilde{M}\cong{S^2}\times{S^2}$;
otherwise $\widetilde{M}\cong{S^2}\times{R^2}$.

\begin{lemma}
Let $q:E\to{F}$ be an $S^2$-bundle over a surface with nonempty boundary.
If $q$ is nontrivial but $q|_{\partial{E}}$ is trivial then there is a
non-separating simple closed curve $\gamma$
in the interior of $F$ such that the restriction of the bundle over
$F\setminus\gamma$ is trivial.
\end{lemma}

\begin{proof}
The bundle is determined by the action of $\pi_1(F)$ on $\pi_2(E)$,
and thus by a class $u\in{H^1(F;\mathbb{F}_2)}$.
Since $u|_{\partial{F}}=0$ and $u\not=0$ 
the Poincar\'e-Lefshetz dual of $u$ is represented
by a simple closed curve $\gamma$ in the interior of $F$,
and $u$ restricts to 0 on $F\setminus\gamma$.
\end{proof}

The restrictions to each fibre of a bundle automorphism of an $S^2$-bundle 
over a connected base must either all preserve the orientation
of the fibre or reverse the orientation of the fibre.
As every $S^2$-orbifold bundle has a fibre-preserving
self-homeomorphism which is the involution on each fibre,
it shall suffice to consider the fibre-orientation-preserving
automorphisms.

\begin{lemma}
Let $q:E\to{F}$ be an $S^2$-bundle over a surface such that 
$q|_{\partial{E}}$ is trivial.
If $\partial{E}$ has boundary components $\{C_i\mid1\leq{i}\leq{d}\}$
for some $d>0$ and if $\phi_i$ is an orientation-preserving 
bundle automorphism of $q|_{C_i}$ for $i<d$
then there is a bundle automorphism $\phi$ of $q$ such that 
$\phi|_{q^{-1}(C_i)}=\phi_i$ for $i<d$.
\end{lemma}

\begin{proof}
We may clearly assume that $d\geq2$.
Suppose first that $q$ is trivial.
We may obtain $F$ by identifying in pairs $2k$ sides of a $(2k+d)$-gon $P$. 
(The remaining sides corresponding to the boundary components $C_i$.)
A bundle automorphism of a trivial $S^2$-bundle over $X$ 
is determined by a map from $X$ to $Homeo(S^2)$.
Let $[\phi_i]$ be the image of $\phi_i$ in $\pi_1(Homeo(S^2))=Z/2Z$,
for $i<d$,
and define $\phi_d$ on $q^{-1}(C_d)$ so that $[\phi_d]=\Sigma_{i<d}[\phi_i]$.
Let $\phi$ be the identity on the images of the other sides of $P$.
Then $[\phi|_{\partial{P}}]=0$ and so $\phi|_{\partial{P}}$ extends across $P$. 
This clearly induces a bundle automorphism $\phi$ of $q$
compatible with the data.

If $q$ is nontrivial let $\gamma$ be a simple closed curve in $F$ 
as in the previous lemma,
and let $N$ be an open regular neighbourhood of $\gamma$.
If $q$ is trivial let $N=\emptyset$.
Then the restriction of $q$ over $F'=F\setminus{N}$ is
trivial, and so $E'=q^{-1}(F')\cong{F'}\times{S^2}$.
If $N\cong\gamma\times(-1,1)$ then $\partial{E'}$ has $d+2$ components;
if $N\cong{Mb}$ and $\partial{E'}$ has $d+1$ components.
In either case, we let $\phi$ be the identity on the new boundary components,
and proceed as before.
\end{proof}

By Lemma 2 the number of components of $\partial{N}$ over which 
the restriction of $p$ is nontrivial is even. 
We may use the following lemmas to simplify the treatment of
such components.
Let $D_{oo}=S^2\setminus3intD^2$ be the ``pair of pants", 
with boundary $\partial{D_{oo}}=C_1\cup{C_1}\cup{C_3}$.

\begin{lemma}
Let $F$ be a compact surface with at least $2$ boundary components
$C$ and $C'$.
Then there is a simple closed curve $\gamma$ in the interior of $F$ 
such that $F=X\cup{Y}$, 
where $X\cong{D_{oo}}$ and $\partial{X}=C\cup{C'}\cup\gamma$.
\end{lemma}

\begin{proof}
Let $\alpha$ be an arc from $C$ to $C'$.
Then we may take $X$ to be a regular
neighbourhood of $C\cup\alpha\cup{C'}$.
\end{proof}

The two exceptional fibres in $E(2,2)$ have regular neighbourhoods
equivalent to $E(2)$. If we delete the interiors of two such
neighbourhoods we obtain the $S^2$-bundle over $D_{oo}$
which is trivial over exactly one component of $\partial{D_{oo}}$.
Since $D_{oo}\simeq{S^1\vee{S^1}}$ this bundle is well-defined
up to isomorphism.

\begin{lemma}
Let $q:E\to{D_{oo}}$ be the $S^2$-bundle which is nontrivial 
over $C_1$ and $C_2$ and trivial over $C_3$.
If $\phi\in {Aut}(q)$ is an automorphism of $q$ let $\phi_i$ be 
the restriction of $\phi$ to $E_i=q^{-1}(C_i)$, and let $b_i$
the underlying self-homeomorphism of $C_i$, for $i\leq3$.
Then
\begin{enumerate}
\item the $b_i$ either all preserve or all reverse orientation;

\item If $\psi$ is an automorphism of $S^2\tilde\times{S^1}$ then there is an
automorphism $\phi$ of $q$ such that $\phi_1=\phi_2=\psi$,
and such that $\phi_3$ extends across $S^2\times{D^2}$;

\item if $\phi\in {Aut}(q)$ then $\phi_1$ and $\phi_2$ are isotopic
if and only if $\phi_3$ extends across $S^2\times{D^2}$;

\item there is a $\phi\in {Aut}(q)$ such that $\phi_1=id$, $\phi_2=\xi$ and
$\phi_3=\tau$.
\end{enumerate}
\end{lemma}

\begin{proof}
Let $L=S^2\times[0,1]^2/\sim$, where $(z,x,0)\sim(a(z),x,1)$
for all $s\in{S^2}$ and $0\leq{x}\leq1$.
Then $L$ is the total space of the nontrivial $S^2$-bundle over the annulus
$A=[0,1]\times{S^1}$, with projection $p_L:L\to{A}$ given by 
$p_L([z,x,y])=(x,e^{2\pi{iy}})$.
The boundary components of $L$ are each homeomorphic to $S^2\tilde\times{S^1}$.
Let $k=(\frac12,1)\in{A}$, $D=\{(x,u)\in{A}\mid d((x,u),K)<\frac14\}$,
$B=A\setminus{D}$ and $E=L\setminus{p_L^{-1}(D)}$.
Then $p_L|_E$ is a model for $q$.

The first assertion is clear, since $D_{oo}$ is orientable.

The automorphism $id_{[0,1]}\times\psi$ of $p_L$ restricts
to an automorphism $\phi$ of $q$ with the desired boundary behaviour.

If $\phi_3$ extends across $S^2\times{D^2}$ then $\phi_1$ and $\phi_2$
together bound an automorphism of $p_L$, and so must be isotopic.
Conversely, if $\phi_1$ and $\phi_2$ are isotopic we may assume that
they are isotopic to the identity, by (2).
The automorphism $\phi$ then extends to an automorphism of $E(2,2)$.
Now $E(2,2)\cup_\tau{S^2}\times{D^2}$
is not homeomorphic to $E(2,2)\cup{S^2}\times{D^2}$.
(See \S4 below).
Therefore $\tau$ does not extend across $E(2,2)$,
and so $\phi_3$ must extend across $S^2\times{D^2}$.

Let $P=(0,-1)$, $Q=(1,-1)$ $R=(\frac34,1)$ and $S=(1,1)$ be points in $B$
and let $B'=B\setminus(PQ\cup{RS})\times(-\varepsilon,\varepsilon)$.
Then $B'\cong{D^2}$, and so the restriction $q'=q|_{B'}$ is trivial.
We may clearly define a bundle automorphism of $q'$
which rotates the fibre once as we go along each of the arcs
corresponding to $\{1\}\times{S^1}$ and $\partial{D}$
and is the identity over the rest of the boundary.
Since the automorphisms agree along the pairs of arcs corresponding to
$PQ$ and $RS$, we obtain the desired automorphism of $q$.
\end{proof}

Let $j:S^2\times{D^2}\to{M}$ be a fibre-preserving embedding of a
closed regular neighbourhood of a regular fibre of $p$,
and let $N$ be the image of $j$.
The {\it Gluck reconstruction} of $p$ is the orbifold bundle 
$p^\tau:M^\tau\to{B}$ with total space 
$M^\tau=M\setminus{intN}\cup_{j\tau}{S^2\times{D^2}}$
and projection given by $p$ on $M\setminus{intN}$
and by projection to the second factor on $S^2\times{D^2}$.

\begin{theorem}
Let $p:M\to{B}$ and $p':M'\to{B}$ be $S^2$-orbifold bundles 
over the same base $B$ and with the same action
$u:\pi_1^{orb}(B)\to\mathbb{Z}^\times$. 
If $\Sigma{B}$ is nonempty then $p'$ is isomorphic to $p$ or $p^\tau$,
and so $M'\cong{M}$ or $M^\tau$.
\end{theorem}

\begin{proof}
The base $B$ has a suborbifold $N$ which contains $\Sigma{B}$ 
and is a disjoint union of copies of regular neighbourhoods 
of reflector curves and copies of $D(2,2)$, by Lemma 2.
If $C$ is a reflector curve, with regular neighbourhood 
$N(C)\cong{J}\times{S^1}$, 
then $p^{-1}(N(C))\cong\mathbb{E}$ or $\mathbb{E}'$, 
while if $D(2,2)\subset{B}$ then $p^{-1}(D(2,2))\cong{E(2,2)}$.

Since $N$ is nonempty and the restrictions of $p$ and $p'$ 
over $B\setminus{N}$ are $S^2$ bundles with the same data they are isomorphic.
Moreover the bundles are trivial over the boundary components of
$B\setminus{N}$.
After composing with a fibrewise involution, if necessary,
we may assume that the bundle isomorphism restricts 
to orientation-preserving homeomorphisms of these boundary components.
Let $R$ be a regular neighbourhood of a regular fibre $S^2$.
Using Lemmas 4 and 6 we may construct a fibre-preserving homeomorphism $h$
from $M\setminus{p^{-1}(R)}$ to $M'\setminus{p'^{-1}(R)}$.
If $h|_{\partial{R}}$ extends across $R$ then $p'\cong{p}$;
otherwise $p'\cong{p^\tau}$.
\end{proof}

If $u$ is nontrivial the standard geometric 4-manifold $M_{st}$
realizing $\pi=\pi_1^{orb}(B)$ is the total space of an orbifold bundle 
$p_{st}$ with regular fibre $S^2$, base $B$ and action $u$.

\begin{cor}[A]
Every $S^2$-orbifold bundle is either geometric or is the
Gluck reconstruction of a standard geometric orbifold bundle.
\qed
\end{cor}

\begin{cor}[B]
If $\Sigma{B}$ contains a reflector curve then 
every $S^2$-orbifold bundle over $B$ is a standard geometric bundle.
\qed
\end{cor}

We may also realize actions with base a non-compact hyperbolic
2-orbifold by geometric orbifold bundles.

\begin{cor}[C]
If $B$ has a nontrivial decomposition into hyperbolic pieces
then $M$ has a proper geometric decomposition.
\qed
\end{cor}

In particular, if $B$ is hyperbolic (and not $T(2,2)$ or $Kb(2,2)$)
then either $M$ is geometric or it has a proper geometric decomposition.

Let $B$ and $\overline{B}$ be 2-orbifolds and 
let $u$ and $\bar{u}$ be actions of $\pi=\pi^{orb}(B)$ and 
$\overline\pi=\pi^{orb}(\overline{B})$ on $Z$ with torsion-free kernels.
An orbifold map $f:B\to\overline{B}$ is {\it compatible with the actions\/} 
$u$ and $\bar{u}$ if it induces an epimorphism $f_*:\pi\to\bar\pi$ 
such that $u=\bar{u}f$.
If $p:\overline{M}\to\overline{B}$ is an $S^2$-orbifold bundle 
realizing $(\overline\pi,\bar{u})$ then the pullback $f^*p$ 
is an $S^2$-orbifold bundle realizing $(\pi,u)$. 
If moreover $f$ is an isomorphism over a non-empty open subset 
of $\overline{B}$ then $(f^*p)^\tau=f^*(p^\tau)$.

In his dissertation Vogt classified $S^2$-orbifold bundles 
over 2-orbifolds with no reflector curves.
While he expected that (in our terminology) Gluck reconstruction 
should change the homeomorphism type of the total space, 
he left this question open \cite{Vo70}.

\section{spherical base orbifold}

If the base orbifold is spherical then it must be one of $S^2$, 
$RP^2$, $S(2,2)$, $\mathbb{D}$ or $\mathbb{D}(2)$, by Lemma 2.
There are two $S^2$-bundle spaces over $S^2$, and four over $RP^2$.
The latter are quotients of $S^2\times{S^2}$ by involutions of the form
$(A,-I)$, where $A\in{GL(3,\mathbb{Z})}$ is a diagonal matrix,
and projection to the quotient of the second factor by
the antipodal map induces the bundle projection.

If $A=diag[-1,-1,1]=R_\pi$ or $diag[1,1,-1]=aR_\pi$
then projection to the first factor induces an orbifold bundle 
(over $S(2,2)$ or $\mathbb{D}$, respectively) with general fibre $S^2$.
The geometric orbifold bundle over $S(2,2)$ 
has total space $E(2,2)\cup{S^2}\times{D^2}$.
It is also the total space of an $S^2$-bundle over $RP^2$.

There is another $S^2$-orbifold bundle over $S(2,2)$,
with total space $RP^4\#_{S^1}RP^4=E(2,2)\cup_\tau{S^2}\times{D^2}$.
(Note that by Lemma 6 there is a bundle automorphism 
of $E(2,2)\setminus{E(2)}$ which is the twist $\tau$ 
on $\partial{E(2,2)}$ and the twist $\xi$ on $\partial{E(2)}$.
Hence $E(2,2)\cup_\tau{S^2}\times{D^2}\cong{E(2)}\cup_\xi{E(2)}$.
The latter model for $RP^4\#_{S^1}RP^4$ is used in \cite{KKR}.)
The total spaces of these two $S^2$-bundles over $S(2,2)$
are not homotopy equivalent,
since the values of the $q$-invariant of \cite{KKR} differ.
Thus $\tau$ does not extend to a homeomorphism of $E(2,2)$.

The $S^2$-orbifold bundle over $\mathbb{D}=S^2/z\sim{aR_\pi(z)}$ 
given by this construction is the unique such bundle, 
by Corollary B of Theorem 7.
(The reflector curve is untwisted.)
The total space is orientable and has $v_2=0$.

Finally, $\mathbb{D}(2)$ is the quotient of $S^2$ by the group $(Z/2Z)^2$
generated by $a$ and $R_\pi$. 
Since these generators commute, $R_\pi$ induces an involution of $RP^2$
which fixes $RP^1$ and a disjoint point.
The corresponding $S^2$-orbifold bundle space
is $S^2\times{S^2}/(x,y)\sim(x,-y)\sim(-x,R_\pi(y))$.
This is again the unique such bundle, 
by Corollary B of Theorem 7.
(The reflector curve is now $u$-twisted.)
It is also the total space of the nontrivial $RP^2$-bundle over $RP^2$.

\section{the $k$-invariant}

If $\pi=\pi_1(M)$ is torsion-free then $c.d.\pi=2$, and so $H^3(\pi;Z^u)=0$. 
Hence $k_1(M)=0$.
Therefore in this section we may assume that $\pi$
has an element $x$ of order 2.

Let $P=P_2(M_{st})$.
The image of $H_4(CP^\infty;\mathbb{F}_2)$ in $H_4(P;\mathbb{F}_2)$
is fixed under the action of $Aut(P)$,
and so $Aut(P)$ acts on this homology group through 
a quotient of order at most 2.
Since $M_{st}$ is geometric $Aut(\pi)$ acts isometrically.
More generally, if $M$ is the total space of an orbifold bundle
then $Aut(\pi)$ acts by orbifold automorphisms of the base.
The antipodal map on the fibres defines a self-homeomorphism which
induces $-1$ on $\pi_2(M)$.
These automorphisms clearly fix $H_4(P;\mathbb{F}_2)$.
Thus it shall be enough to consider the action of the subgroup 
of $Aut(P)$ which acts trivially on $\pi_1$ and $\pi_2$.
Since $P$ is a connected cell-complex with $\pi_i(P)=0$ for $i>2$
this subgroup is isomorphic to $H^2(\pi;Z^u)$ \cite{Ru92}.

\begin{theorem}
Let $M_o=M_{st}\setminus{intD^4}$ be the complement of an open disc in $M_{st}$.
Then $M_{st}^\tau\simeq{M_o\cup_fD^4}$ for some $f:S^3\to{M_o}$.
\end{theorem}

\begin{proof}
Since
$S^2\times{D^2}=(D^2\times{D^2})\cup(D^2\times{D^2})=(D^2\times{D^2})\cup{D^4}$,
we may obtain each of $M_{st}$ and $M_{st}^\tau$ from $M_{st}\setminus{N}$ 
(up to homotopy) by first adding a 2-cell and then a 4-cell.
The attaching maps for the 2-cells are the inclusions $u\mapsto(1,u)$
and $u\mapsto(u,u)$ of $S^1$ into $\partial{N}=S^2\times{S^1}$, 
respectively.
Since these are clearly homotopic,
$M_{st}^\tau$ may be obtained from $M_{st}$ 
by changing the attaching map for the top cell
of $M_{st}=M_o\cup{D^4}$.
\end{proof}

(It can be shown that the attaching maps differ
by the image of the Hopf map $\eta$ in $\pi_3(M_o)$.)

\begin{cor}
The inclusions of $M_o$ into $M_{st}$ and $M_{st}^\tau$
induce isomorphisms of cohomology in degrees $\leq3$.
\qed
\end{cor}

This theorem also implies that $P_2(M_{st}^\tau)\simeq{P_2(M_{st})}$,
since each may be constructed by adjoining cells to $M_o$ to
kill the higher homotopy.
However the Corollary of Theorem 10 below is stronger, 
in that it does not assume the manifolds under consideration 
are $S^2$-orbifold bundle spaces.
If $M$ is {\it any} closed 4-manifold with $\widetilde{M}\simeq{S^2}$
then the $u$-twisted Bockstein $\beta^u$ maps $H^2(\pi;\mathbb{F}_2)$
onto $H^3(\pi;Z^u)$, and the restriction of $k_1(M)$
to each subgroup of order 2 in $\pi$ is nontrivial,
by Lemma 10.4 of \cite{Hi}.
On looking at the structure of such groups and applying
Mayer-Vietoris arguments to compute these cohomology groups,
we can show that there is only one possible $k$-invariant.

\begin{lemma}
Let $\alpha=*^kZ/2Z=
\langle{x_i, 1\leq{i}\leq{k}}\mid{x_i^2=1~\forall~i}\rangle$ 
and let $u(x_i)=-1$ for all $i$.
Then restriction from $\alpha$ to $\phi=\mathrm{Ker}(u)$ 
induces an epimorphism from $H^1(\alpha;Z^u)$ to $H^1(\phi;Z)$.
\end{lemma}

\begin{proof}
Let $x=x_1$ and $y_i=x_1x_i$ for all $i>1$.
Then $\phi=\mathrm{Ker}(u)$ is free with basis $\{y_2,\dots,y_k\}$
and so $\alpha\cong{F(k-1)\rtimes{Z/2Z}}$.

If $k=2$ then $\alpha$ is the infinite dihedral group $D$ and
the lemma follows by direct calculation with resolutions.
In general, the subgroup $D_i$ generated by $x$ and $y_i$ is an
infinite dihedral group, and is a retract of $\alpha$.
The retraction is compatible with $u$,
and so restriction maps $H^1(\alpha;Z^u)$ onto $H^1(D_i;Z^u)$.
Hence restriction maps $H^1(\alpha;Z^u)$ onto each summand
$H^1(\langle{y_i}\rangle;Z)$ of $H^1(\phi;Z)$, and the result follows.
\end{proof}

In particular, if $k$ is even then $z=\Pi{x_i}$ generates 
a free factor of $\phi$, and restriction 
maps $H^1(\alpha;Z^u)$ onto $H^1(\langle{z}\rangle;Z)$.

Let $S(2_k)$ be the sphere with $k$ cone points of order 2.

\begin{theorem}
Let $B$ be an aspherical $2$-orbifold, 
and let $u:\pi=\pi_1^{orb}(B)\to\mathbb{Z}^\times$ be an
epimorphism with torsion-free kernel $\kappa$. 
Suppose that $\Sigma{B}\not=\emptyset$, and that
$B$ has $r$ reflector curves and $k$ cone points.
Then $H^2(\pi;Z^u)\cong(Z/2Z)^r$ if $k>0$ and
$H^2(\pi;Z^u)\cong{Z}\oplus(Z/2Z)^{r-1}$ if $k=0$.
In all cases $\beta^u(U^2)$ is the unique element of $H^3(\pi;Z^u)$    
which restricts non-trivially to each subgroup of order $2$.
\end{theorem}

\begin{proof}
Suppose first that $B$ has no reflector curves.
Then $B$ is the connected sum of a closed surface $G$
with $S(2_k)$, and $k$ is even, by Lemma 2.
If $B=S(2_k)$ then $k\geq4$, since $B$ is aspherical.
Hence $\pi\cong\mu*_Z\nu$, 
where $\mu=*^{k-2}Z/2Z$ and $\nu=Z/2Z*Z/2Z$
are generated by cone point involutions.
Otherwise $\pi\cong\mu*_Z\nu$,
where $\mu=*^kZ/2Z$ and $\nu=\pi_1(G\setminus{D^2})$ 
is a non-trivial free group.
Every non-trivial element of finite order in such a generalized 
free product must be conjugate to one of the involutions.
In each case a generator of the amalgamating subgroup is identified with 
the product of the involutions which generate the factors of $\mu$
and which is in $\phi=\mathrm{Ker}(u|_\mu)$.

Restriction from $\mu$ to $Z$ induces an epimorphism from
$H^1(\mu;Z^u)$ to $H^1(Z;Z)$, by Lemma 9,
and so 
\[H^2(\pi;Z^u)\cong{H^2(\mu;Z^u)}\oplus{H^2(\nu;Z^u)}=0,\]
by the Mayer-Vietoris sequence with coefficients $Z^u$.
Similarly,
\[H^2(\pi;\mathbb{F}_2)\cong
{H^2(\mu;\mathbb{F}_2)}\oplus{H^2(\nu;\mathbb{F}_2)},\]
by the Mayer-Vietoris sequence with coefficients $\mathbb{F}_2$.
Let $e_i\in{H^2(\pi;\mathbb{F}_2)}$
$=Hom(H_2(\pi;\mathbb{F}_2),\mathbb{F}_2)$ 
correspond to restriction to the $i$th cone point.
Then $\{e_1,\dots,e_{2g+2}\}$ forms a basis for 
$H^2(\pi;\mathbb{F}_2)\cong\mathbb{F}_2^{2g+2}$,
and $\Sigma{e_i}$ is clearly the only element with nonzero
restriction to all the cone point involutions.
Since $H^2(\pi;Z^u)=0$ the $u$-twisted Bockstein maps
$H^2(\pi;\mathbb{F}_2)$ isomorphically onto $H^3(\pi;Z^u)$,
and so there is an unique possible $k$-invariant.

Suppose now that $r>0$.
Then $B=r\mathbb{J}\cup{B_o}$, where $B_o$ is a connected 2-orbifold 
with $r$ boundary components and $k$ cone points.
Hence $\pi=\pi\mathcal{G}$,
where $\mathcal{G}$ is a graph of groups with underlying graph
a tree having one vertex of valency $r$ with group 
$\nu=\pi_1^{orb}(B_o)$, $r$ terminal vertices,
with groups $\gamma_i\cong\pi_1^{orb}(\mathbb{J})={Z}\oplus{Z/2Z}$,
and $r$ edge groups $\omega_i\cong{Z}$.
If $k>0$ then restriction maps $H^1(\nu;Z^u)$ onto $\oplus{H^1(\omega_i;Z)}$
and then $H^2(\pi;Z^u)\cong\oplus{H^2(\gamma_i;Z^u)}\cong{Z/2Z}^r$.
However if $k=0$ then $H^2(\pi;Z^u)\cong{Z}\oplus(Z/2Z)^{r-1}$.

The Mayer-Vietoris sequence with coefficients $\mathbb{F}_2$
gives an isomorphism 
$H^2(\pi;\mathbb{F}_2)\cong{H^2(\nu;\mathbb{F}_2)}\oplus
(H^2(Z\oplus{Z/2Z};\mathbb{F}_2))^r\cong\mathbb{F}_2^{2r+k}$.
The generator of the second summand of $H^2(Z\oplus{Z/2Z};\mathbb{F}_2)$
is in the image of reduction modulo $(2)$ from $H^2(Z\oplus{Z/2Z};Z^u)$, 
and so is in the kernel of $\beta^u$.
Therefore the image of $\beta^u$ has a basis corresponding to the 
cone points and reflector curves,
and we again find an unique $k$-invariant.
Since $\beta^u(U^2)$ restricts to the generator of $H^3(Z/2Z;Z^u)$ at
each involution in $\pi$, 
we must have $k_1(M)=\beta^u(U^2)$.
\end{proof}

\begin{cor}
If $M$ is a closed $4$-manifold with $\pi_2(M)\cong{Z}$
and $\pi_1(M)\cong\pi^{orb}(B)$ then $P_2(M)\simeq{P_2(M_{st})}$, 
where $M_{st}$ is the standard geometric $4$-manifold with the
same fundamental group.
\qed
\end{cor}

\section{the image of $[M]$ in $H_4(P_2(M);\mathbb{F}_2)$}

As in \cite{Hi09} it is useful to begin this section
by considering first the simpler case when $u$ is trivial.
The group $\pi$ is then a $PD_2$-group, and so $k_1(M)=0$.
Let $F$ be a closed surface with $\pi_1(F)=\pi$,
and let $P={CP^\infty}\times{F}\simeq\Omega{K(Z,3)}\times{F}$.
The natural inclusion $f_{st}:M_{st}=S^2\times{F}\to{P}$ is 3-connected,
and so it is the second stage of the Postnikov tower for $M_{st}$.

The nontrivial bundle space with this group and action
is the Gluck reconstruction $M_{st}^\tau$.
We may assume that the neighbourhood $N$ of a fibre is a 
product $S^2\times{D^2}$, where $D^2\subset{F}$.
Let $h:M^\tau\to{CP^2\times{F}}\subset{P}$ be the map defined by
$h(m)=f_{st}(m)$ for all $m\in{M\setminus{N}}$
and $h([z_0:z_1],d)=([dz_0:z_1:(1-|d|)z_0],d)$ for all $[z_0:z_1]\in{S^2=CP^1}$
and $d\in{D^2}$.
(The two definitions agree on $S^2\times{S^1}$,
since $\tau([z_0:z_1],u)=([uz_0:z_1],u)$ for $u\in{S^1}$.)
Then $h$ is 3-connected, and so is the second stage of 
the Postnikov tower for $M_{st}^\tau$.

By the K\"unneth Theorem, 
\[H_4(P;\mathbb{F}_2)\cong{H_4(CP^\infty;\mathbb{F}_2)}\oplus
(H_2(CP^\infty;\mathbb{F}_2)\otimes{H_2(F;\mathbb{F}_2))}\cong\mathbb{F}_2^2.\]
Homotopy classes of self-maps of $P$ which induce the identity on $\pi$ 
and $\pi_2$ are represented by maps $(c,f)\mapsto (c.s(f),f)$,
where $s:F\to\Omega{K(Z,3)}$ and we use the loop space multiplication 
on $\Omega{K(Z,3)}$. 
It is not hard to see that these act trivially on $H_4(P;\mathbb{F}_2)$.
Since automorphisms of $\pi$ and $\pi_2$ are realized 
by self-homeomorphisms of $F$ and $CP^\infty$, respectively,
$Aut(P)$ acts trivially on $H_4(P;\mathbb{F}_2)$.

Let $q:P\to{CP^\infty}$ be the projection to the first factor.
Then $qf_{st}$ factors through the inclusion of $CP^1$, and so has degree 0.
On the other hand, 
if $(w,d)$ is in the open subset $U=\mathbb{C}\times{intD^2}$ 
with $z_0\not=0$ and $|d|<1$ then $qh(w,d)=[d:w:1-|d|]$, 
and ${(qh)^{-1}([a:b:1])}=(b/(1+|a|),a/(1+|a|))$.
Hence $qh$ maps $U$ bijectively onto the dense open subset 
$CP^2\setminus{CP^1}$,
and collapses $M_{st}^\tau\setminus{h(U)}=M\setminus{intN}$ onto $CP^1$.
Therefore $qh:M_{st}^\tau\to{CP^2}$ has degree 1.
Thus the images of $[M_{st}]$ and $[M_{st}^\tau]$ 
in $H_4(P_2(M);\mathbb{F}_2)$ are not equivalent under the action of $Aut(P)$.

This is not surprising, as $v_2(M_{st})=0$,
but twisting the neighbourhood of a regular fibre changes the 
{\it mod}-(2) self-intersection number of a section to the bundle,
and so $v_2(M_{st}^\tau)\not=0$.

If $M$ is an $S^2$-orbifold bundle space with exceptional fibres 
then the image of a regular fibre in $H_2(M;\mathbb{F}_2$) is trivial, 
since the inclusion factors through the covering $S^2\to{RP^2}$, 
up to homotopy.
Therefore the {\it mod}-(2) Hurewicz homomorphism is trivial,
and Gluck reconstruction does not change the {\it mod}-(2) 
self-intersection pairing.
In particular, $H^2(\pi;\mathbb{F}_2)\cong {H^2(M;\mathbb{F}_2)}$,
and $v_2(M_{st}^\tau)=v_2(M_{st})$.

Although we cannot expect to detect the effect of twisting through the Wu class,
we may adapt the argument above to $S^2$-orbifold bundles with $u\not=1$.
Then
\[K(\pi,1)\simeq{S^\infty}\times{K(\kappa,1)}/(s,k)\sim(-s,\zeta(k)).\]
(If $\pi$ is torsion-free we do not need the $S^\infty$ factor.)
The antipodal map of $CP^1=S^2$ extends to involutions on $CP^n$ 
given by 
\[[z_0:z_1:z_2:\dots:z_n]\mapsto
[-\overline{z_1}:\overline{z_0}:\overline{z_2}:\dots:\overline{z_n}].\]
(Here only the first two harmonic coordinates change position or sign.)
Since these are compatible with the inclusions of $CP^n$ into $CP^{n+1}$ 
given by $[z_0:\dots:z_n]\mapsto[z_0:\dots:z_n:0]$,
they give rise to an involution $\sigma$ on $CP^\infty=\varinjlim{CP^n}$.
Let 
\[P=CP^\infty\times{S^\infty}\times{K(\kappa,1)}/
(z,s,k)\sim(\sigma(z),-s,\zeta(k)).\]
Then $\pi_1(P)\cong\pi$, $\pi_2(P)\cong{Z^u}$ and $\pi_j(P)=0$ for $j>2$.
We shall exclude the case of $RP^2$-bundle spaces, with
$\pi\cong\kappa\times{Z/2Z}$, as these are well understood.
(The self-intersection number argument does apply in this case.)

\begin{theorem}
Let $\pi$ be a group with an epimorphism $u:\pi\to{Z/2Z}$
such that $\kappa=\mathrm{Ker}(u)$ is a $PD_2$-group, and 
suppose that $\pi$ is not a direct product $\kappa\times{Z/2Z}$.
Let $M_{st}$ be the standard geometric $4$-manifold corresponding 
to $(\pi,u)$ and $P=P_2(M_{st})$.
Then the images of $[M_{st}]$ and $[M_{st}^\tau]$ in $H_4(P;\mathbb{F}_2)$ 
are distinct.
\end{theorem}

\begin{proof}
The diagonal map from $S^2$ to ${S^2}\times{S^2}=CP^1\times{S^2}$
determines a 3-connected map $f_{st}:M_{st}\to{P}$ 
by $f_{st}([s,k])=[s,s,k]$.
This is the second stage of the Postnikov tower for $M_{st}$, 
and embeds $M_{st}$ as a submanifold of
${CP^1\times{S^2}\times{K(\kappa,1)}/\sim}$ in $P$.
We again have $H_4(P;\mathbb{F}_2)\cong\mathbb{F}_2^2$,
with generators the images of $[M_{st}]$ and $[CP^2]$.

The projection of $CP^\infty\times{S^\infty}\times{K(\kappa,1)}$
onto its first two factors induces a map 
$g:P\to{Q}=CP^\infty\times{S^\infty}/(z,s)\sim(\sigma(z),-s)$
which is in fact a bundle projection with fibre $K(\kappa,1)$.
Since $gf_{st}$ factors through $S^2$ the image of $[M_{st}]$
in $H_4(Q;\mathbb{F}_2)$ is trivial.
 
Since $\pi$ is not a direct product,
$M_{st}$ is the total space of an $S^2$-orbifold bundle $p_{st}$.
Let $v:S^2\times{D^2}\to{V}\subset{M_{st}}$ be a fibre-preserving
homeomorphism onto a regular neighbourhood of a regular fibre.
Since $V$ is 1-connected $f_{st}|_V$ factors through 
${CP^\infty\times{S^\infty}\times{K(\kappa,1)}}$.
Let $f_1$ and $f_2$ be the composites of a fixed lift of 
$f_{st}v\tau:S^2\times{S^1}\to{P}$ with the projections to $CP^\infty$ 
and $S^\infty$, respectively.
Let $F_1$ be the extension of $f_2$ given by
\[F_2([z_0:z_1],d)=[dz_0:z_1:(1-|d|)z_0]\]
for all $[z_0:z_1]\in{S^2}=CP^1$ and $d\in{D^2}$.
Since $f_2$ maps $S^2\times{S^1}$ to $S^2$ it is nullhomotopic in $S^3$,
and so extends to a map $F_2:S^2\times{D^2}\to{S^3}$.
Then the map $F:M_{st}^\tau\to{P}$ given by $f_{st}$ 
on $M_{st}\setminus{N}$ and $F(s,d)=[F_1(s),F_2(s),d]$ 
for all $(s,d)\in{S^2\times{D^2}}$ is 3-connected,
and so it is the second stage of the Postnikov tower for $M_{st}^\tau$.

Now $F_1$ maps the open subset $U=\mathbb{C}\times{intD^2}$ 
with $z_0\not=0$ bijectively onto its image in $CP^2$,
and maps $V$ onto $CP^2$.
Let $\Delta$ be the image of $CP^1$ under the diagonal embedding in
$CP^1\times{CP^1}\subset{CP^2\times{S^3}}$.
Then $(F_1,F_2)$ carries $[V,\partial{V}]$ to the image of $[CP^2,CP^1]$ in
$H_4(CP^2\times{S^3},\Delta;\mathbb{F}_2)$.
The image of $[V,\partial{V}]$ generates $H_4(M,M\setminus{U};\mathbb{F}_2)$.
A diagram chase now shows that $[M^\tau]$ and $[CP^2]$ have the same
image in $H_4(Q;\mathbb{F}_2)$, and so $[M^\tau]\not=[M]$
in $H_4(P_2(M);\mathbb{F}_2)$.
\end{proof}

It remains to consider the action of $Aut(P)$.
Since $M$ is geometric $Aut(\pi)$ acts isometrically.
The antipodal map on the fibres defines a self-homeomorphism which
induces $-1$ on $\pi_2(M)$.
These automorphisms clearly fix $H_4(P;\mathbb{F}_2)$.
Thus it is enough to consider the action of $G=H^2(\pi;Z^u)$
on $H^2(\pi;Z^u)$.

\begin{cor}
Every $4$-manifold realizing $(\pi,u)$ is homotopy equivalent 
to $M$ or $M^\tau$.
If $B=X/\pi$ has no reflector curves then $M^\tau\not\simeq{M}$.
\end{cor}

\begin{proof}
The first assertion holds since the image of the fundamental class
in $H_4(P_2(M);\mathbb{F}_2)$ must generate {\it mod} $[CP^2]$,
and so be $[M]$ or $[M]+[CP^2]$.

If $B$ is nonsingular then Gluck reconstruction changes the
self-intersection of a section,
and hence changes the Wu class $v_2(M)$. 
If $B$ has cone points but no reflector curves then $H^2(\pi;Z^u)=0$,
by Theorem 10, and so $M^\tau\not\simeq{M}$, by Theorem 11.
\end{proof} 

Is there a more explicit invariant?
The $q$-invariant of \cite{KKR} is 0 for every
orbifold bundle with regular fibre $S^2$ over an aspherical base.

A closed 4-manifold $M$ is strongly minimal if 
the equivariant intersection pairing on $\pi_2(M)$ is 0.
Every group $G$ with $c.d.G\leq2$ is the fundamental group of a
strongly minimal 4-manifold, and every closed 4-manifold
with fundamental group $G$ admits a 2-connected degree-1 
map to a strongly minimal 4-manifold \cite{Hi09}.
However, if we allow torsion but assume that $v.c.d.G=2$ and
$G$ has one end then $\pi\cong\kappa\rtimes{Z/2Z}$,
with $\kappa$ a $PD_2$-group, by Theorem 4 of \cite{Hi09}.
When does a closed 4-manifold $N$ with $\pi_1(N)\cong\kappa\rtimes{Z/2Z}$
admit a 2-connected degree-1 map to an $RP^2$-bundle space or to
an $S^2$-orbifold bundle space?

\section{the main result}

We may now summarize our results in the following theorem.

\begin{theorem}
Let $M$ be a closed $4$-manifold with $\pi_2(M)\cong{Z}$,
and let $\kappa=\mathrm{Ker}(u)$, 
where $u:\pi=\pi_1(M)\to{Aut(\pi_2(M))}=\mathbb{Z}^\times$ 
is the natural action.
Then
\begin{enumerate}
\item{if} $\pi=1$ then $M\simeq{CP^2}$;

\item{if} $\pi\cong\kappa\times{Z/2Z}$ then $M$ is homotopy equivalent to the
total space of an $RP^2$-bundle over an aspherical surface
$F\simeq{K(\kappa,1)}$;

\item{if} $\pi\not=1$ and $\pi\not\cong\kappa\times{Z/2Z}$  
then $M$ is an $S^2$-orbifold bundle space over an aspherical
$2$-orbifold $B$ with $\pi^{orb}(B)\cong\pi$.
If $B$ has a reflector curve then $M\simeq{M_{st}}$;
otherwise there are two homotopy types.
\end{enumerate}
\end{theorem}

\begin{proof}
If $\pi=1$ then $P_2(M)\simeq{CP^\infty}$, 
and the classifying map $f_M:M\to{P_2(M)}$ 
factors through $CP^2$, by general position.
This map induces isomorphisms on cohomology, 
by the nonsingularity of Poincar\'e duality,
and so is a homotopy equivalence.

If $\pi\cong\kappa\times{Z/2Z}$ then $M$ is homotopy equivalent to the
total space of an $RP^2$-bundle over an aspherical surface $F$, 
by Theorem 5.16 of \cite{Hi}.
Clearly $\pi_1(F)\cong\pi$.

If $\pi$ is nontrivial and not a product with $Z/2Z$ then
$k_1(M)$ is determined by $(\pi,u)$, by Theorem 9, and so
there are at most two possible homotopy types,
by Theorem 10.6 of \cite{Hi}.
These are represented by the $S^2$-orbifold bundle spaces 
$M_{st}$ and $M_{st}^\tau$, by Theorem 11.
If moreover $B$ has a reflector curve 
then $M_{st}^\tau$ and $M_{st}$ are diffeomorphic, by Corollary B of
Theorem 7.
Otherwise, $H^2(\pi;Z^u)=0$ and so these orbifold bundle spaces 
are not homotopy equivalent.
\end{proof}

\begin{cor} [A]
Let $M_\kappa$ be the double cover associated to $\kappa$.
If $\pi\not=1$ and $\pi\not\cong\kappa\times{Z/2Z}$
then $M_\kappa\simeq{S^2}\times{K(\kappa,1)}$.
\end{cor}

\begin{proof}
The double cover of $M_{st}$ is ${S^2}\times{K(\kappa,1)}$,
and the double cover of $M_{st}^\tau$ may be obtained from 
this by two Gluck reconstructions.
Hence these covers are homeomorphic.
The second assertion follows.
\end{proof}

The quotient of the total space of any $S^2$-bundle over 
a closed surface $F$ by the fibrewise antipodal involution
is an $RP^2$-bundle over $F$.
Thus the corollary fails if $\pi\cong\kappa\times{Z/2Z}$.

\begin{cor} [B]
If $M$ is orientable and $\pi$ has torsion then $M\simeq{M_{st}}$.
\end{cor}

\begin{proof}
The double cover $M_\kappa$ is an $S^2$-bundle over a surface $F$.
Since $M$ is orientable and $\kappa$ acts trivially,
$F$ must also be orientable and the covering involution of $F$ 
over the base orbifold $B$ must be orientation-reversing.
Since $\pi$ has torsion $\Sigma{B}$ is a non-empty union of reflector curves,
by Lemma 2.
\end{proof}

If $M$ is orientable then the base $B$ is non-orientable.
In fact all $S^2$-orbifold spaces over non-orientable bases are geometric,
by the next result.

\begin{theorem}
Let $B$ be a $\mathbb{X}^2$-orbifold 
and let $u:\pi=\pi^{orb}(B)\to{Z/2Z}$ be an epimorphism 
with torsion-free kernel $\kappa$. 
Then $M_{st}^\tau$ is geometric if and only if either 
$B$ has a reflector curve or $\pi$ is not generated by involutions.
\end{theorem}

\begin{proof}
If $\pi$ is torsion-free then all $S^2$-bundle spaces over $B$ are geometric, 
by Theorems 10.8 and 10.9 of \cite{Hi},
while if $\Sigma{B}$ has a reflector curve then $M_{st}^\tau\cong{M_{st}}$,
by Theorem 7.
Therefore we may assume that $\Sigma{B}$
is a non-empty finite set of cone points of order $2$.

If $B$ has no reflector curves and $\pi=\pi^{orb}(B)$ 
is generated by involutions then $B$ 
is the quotient of an orientable surface  
by the hyperelliptic involution.
As involutions have fixed points in $R^2$, 
they must act without fixed points on $S^2$.
Therefore every geometric 4-manifold 
with group $\pi$ is diffeomorphic to $M_{st}$,
and so $M_{st}^\tau$ is not geometric.

If $\pi$ is not generated by involutions
then $B\cong{S((2)_{2k})}\#{G}$, 
where $G$ is a closed surface other than $S^2$.
The action $u$ is trivial on the separating curve of the connected sum,
and so defines an action $u_G$ of $\pi_1(G)$ on $Z$.
The Gluck reconstruction of the standard $S^2$-orbifold bundle over $B$
may be achieved by modifying the $S^2$-bundle over $G$.
If $G$ is aspherical the Gluck reconstruction of the standard bundle over $G$
again has geometric total space, 
and the two bundles realizing the action $u_G$ are 
distinguished by the representation $\rho$ of $\pi_1(G)$ in $O(3)$,
as in Theorems 10.8 and 10.9 of \cite{Hi}.
We may clearly modify the standard representation of $\pi=\pi^{orb}(B)$
to show that $M^\tau$ is also geometric.
Otherwise, $G=RP^2$ and $B\cong{S((2)_{2(k-1)})}\#{P(2,2)}$,
and a similar argument applies.
\end{proof}

\section{the second wu class}

If $M$ is an $S^2$-bundle space (with $\pi$ torsion-free) 
Gluck reconstruction changes the second Wu class $v_2(M)$. 
Similarly, if $M$ is an $RP^2$-bundle space we may change 
$v_2(M)$ by reattaching a product neighbourhood of a fibre.
However we shall show here that $v_2(M)$ is determined by $\pi$
if $M$ is an $S^2$-orbifold bundle space and the base orbifold 
has singularities. 

If $\widetilde{M}\simeq{S^2}$ and $x\in\pi$ has order 2 
then the generator of $\pi_2(M)$ factors through 
$\widetilde{M}/\langle{x}\rangle\simeq{RP^2}$, 
and so the {\it mod}-(2) Hurewicz homomorphism is trivial.
Hence $H^i(\pi;\mathbb{F}_2)\cong{H^i(M;\mathbb{F}_2)}$ for $i\leq2$.

\begin{lemma}
The restriction $Res_\pi^\kappa:H^2(\pi;\mathbb{F}_2)\to
{H^2(\kappa;\mathbb{F}_2)}=\mathbb{F}_2$
is surjective, 
and cup-product with $U$ maps $H^1(\pi;\mathbb{F}_2)$ 
onto $\mathrm{Ker}(Res_\pi^\kappa)$.
\end{lemma}

\begin{proof}
Let $\theta$ be the automorphism of $H^1(\kappa;\mathbb{F}_2)$ 
given by $\theta(A)(k)=A(xkx)$ for all $A\in{H^1(\pi;\mathbb{F}_2)}$ 
and $k\in\kappa$.
Let $r=\mathrm{dim}_{\mathbb{F}_2}\mathrm{Ker}(\theta+1)$
and $s=\mathrm{dim}_{\mathbb{F}_2}\mathrm{Im}(\theta+1)$.
Then 
$\mathrm{dim}_{\mathbb{F}_2}H^1(Z/2Z;H^1(\kappa;\mathbb{F}_2))=r-s$
and $\beta_1(\kappa;\mathbb{F}_2)=r+s$.
It follows from the LHS spectral sequence that
$\beta_1(\pi;\mathbb{F}_2)=1+r$ and 
$\beta_2(\pi;\mathbb{F}_2)=1+r-s+\delta$,
where $\delta=\mathrm{dim}_{\mathbb{F}_2}\mathrm{Im}(Res_\pi^\kappa)\leq1$.
Since $\chi(M)=2-2\beta_1(\pi;\mathbb{F}_2)+\beta_2(\pi;\mathbb{F}_2)$
and also $\chi(M)=\chi(\kappa)=2-\beta_1(\kappa;\mathbb{F}_2)$,
we see that in fact $\delta=1$.
Therefore $Res_\pi^\kappa$ is surjective.

Certainly $Res_\pi^\kappa(U\cup{A})=0$ for all $A\in{H^1(\pi;\mathbb{F}_2)}$,
and $U^2\not=0$.
Suppose that $A\in{H^1(\pi;\mathbb{F}_2)}$ is such that $A(x)=0$.
If $U\cup{A}=0$ there is a function $f:\pi\to\mathbb{F}_2$ such that
$U(g)A(h)=f(g)+f(h)+f(gh)$ for all $g,h\in\pi$.
If $g\in\kappa$ then $U(g)=0$ and so $f|_\kappa$ is a homomorphism.
Taking $g=x$ we have $A(h)=f(x)+f(h)+f(xh)$, for all $h\in\pi$,
while taking $h=x$ we have $f(gx)=f(g)+f(x)$ for all $g\in\pi$.
In particular, $f(xhx)=f(xh)+f(x)$, for all $h\in\pi$.
Therefore $A(h)=f(h)+f(xhx)$, for all $h\in\pi$,
and so $A\in\mathrm{Im}(\theta+1)$. 
Thus $\mathrm{dim}_{\mathbb{F}_2}\mathrm{Ker}(U\cup-)\leq{s}$,
and so the image of cup-product with $U$ has rank
at least $r-s+1=\mathrm{dim}_{\mathbb{F}_2}\mathrm{Ker}(Res_\pi^\kappa)$.
\end{proof}

If $r>s$ then there are classes $A,B\in{H^1(\pi;\mathbb{F}_2)}$ 
such that $Res_\pi^\kappa(A\cup{B})\not=0$.
However if $r=s$ then $U\cup{H^1(\pi;\mathbb{F}_2)}=\langle{U^2}\rangle$.
The image of $U^3$ in $H^3(M_{st};\mathbb{F}_2)$ is 0, 
since $H^3(RP^2;\mathbb{F}_2)=0$.
Therefore $U^3$ also has image 0 in $H^3(M_{st}^\tau;\mathbb{F}_2)$,
by the Corollary to Theorem 8.
(Can we see this for any 4-manifold $M$ with $\widetilde{M}\simeq{S^2}$
more directly, without invoking Theorem 12?) 
 
\begin{theorem}
Let $p:M\to{B}$ be an $S^2$-orbifold bundle, and suppose that
$\Sigma{B}\not=\emptyset$. 
\begin{enumerate}
\item{If} $Res_\pi^\kappa(A)^2=0$ for all $A\in{H^1(\pi;\mathbb{F}_2)}$ 
and $B$ has a reflector curve then $v_2(M)=0$;
\item{if} $Res_\pi^\kappa(A)^2=0$ for all $A\in{H^1(\pi;\mathbb{F}_2)}$
and $B$ has a cone point of order $2$ then $v_2(M)=U^2$;
\item{if} $Res_\pi^\kappa(A)^2\not=0$ for some $A\in{H^1(\pi;\mathbb{F}_2)}$
then $v_2(M)=UW$. 
\end{enumerate}
\end{theorem}

\begin{proof}
If $A\in{H^1(\pi;\mathbb{F}_2)}$ then $U^2\cup{UA}=0$, since $U^3=0$.
Let $\sigma:RP^2\to{M}$ be an exceptional fibre.
Then $U^2(\sigma_*[RP^2])=1$, 
and so $H^2(\pi;\mathbb{F}_2)$ is generated by 
$U\cup{H^1(\pi;\mathbb{F}_2)}$ and
the Poincar\'e dual of $\sigma_*[RP^2]$, by Lemma 13.
(In particular, 
$\phi$ has nonzero restriction to $H^2(\kappa;\mathbb{F}_2)$.)

If $\sigma$ is a fibre over a point on a reflector curve of $B$ 
then it has self-intersection 0, and $\phi^2=0$.
If $\sigma$ is a fibre over a cone point of order 2
it has a regular neighbourhood isomorphic to $E(2)$. 
Let $\sigma_t[\pm{s}]=[s,(tx,ty)]\in{E(2)}$
for $s=(x,y,z)\in{S^2}$ and $|t|<1$.
Then $\sigma=\sigma_0$ and $\sigma_t$ is an isotopy of 
embeddings with $\sigma_t.\sigma_0=1$ if $t\not=0$.
In this case $\sigma$ has self-intersection 1, and $\phi^2\not=0$.

Suppose that $Res_\pi^\kappa(A^2)=0$
for all $A\in{H^1(\pi;\mathbb{F}_2)}$. 
Then $v_2(M)=0$ (if $\phi^2=0$) or $U^2$ (otherwise),
by the nonsingularity of Poincar\'e duality.
The first two assertions now follow.

On the other hand, 
if there is an $A\in{H^1(\pi;\mathbb{F}_2)}$ with 
$Res_\pi^\kappa(A)^2\not=0$ then $H^2(\pi;\mathbb{F}_2)$ 
is generated by $U\cup{H^1(\pi;\mathbb{F}_2)}$ and $A^2$.
Since $Res_\pi^\kappa(W)=w_1(\kappa)$,
we then have $Res_\pi^\kappa(A)^2=Res_\pi^\kappa(AW)$.
In particular, $w_1(\kappa)\not=0$ and so $W\not=0$ or $U$.
Poincar\'e duality now gives 
$v_2(M)=UW+\delta{U^2}$, where $\delta=0$ (if $A^3=0$
or if $UA^3$ and $A^4$ are nonzero) or 1 (otherwise).
We may determine $\delta$ by passing to suitable 2-fold covers.

If $B$ has a reflector curve then so does the 2-fold cover $B^+$ 
associated to $\mathrm{ker}(W)$, and so 
$v_2(M^+)=\delta{U^2}$ must be 0, by (1).
If $B$ has cone points we consider instead the
covering spaces $M_V$ and $B_V$ on which $U=W$.
The orbifold $B_V$ now has cone points, and so $v_2(M_V)=(1+\delta)U^2$
must be $U^2$, by (2).
In each case $\delta$ must be 0, and so $v_2(M)=UW$.
\end{proof}

If $\kappa$ is orientable then $Res_\pi^\kappa(A^2)=0$
for all $A\in{H^1(\pi;\mathbb{F}_2)}$.
However the converse is false: if $\pi=Z*_ZD=\pi^{orb}(P(2,2))$
then $\kappa=\pi_1(Kb)$ is non-orientable but
$Res_\pi^\kappa(A^2)=0$ for all $A\in{H^1(\pi;\mathbb{F}_2)}$.
Is it easy to see directly that if $B$ 
has both cone points and reflector curves then 
this condition does not hold?

Whereas regular fibres in an $S^2$-orbifold bundle 
over a connected base are isotopic, 
exceptional fibres over distinct components 
of the singular locus of $B$ are usually not even homologous.
An arc $\gamma$ in $B$ connecting two such components is
in fact a reflector interval, and the
restriction of the fibration over $\gamma$ 
has total space $RP^3\#RP^3$.
The fibres over the reflector points represent independent
generators of $H_2(RP^3\#RP^3;\mathbb{F}_2)$.
Thus it should not be surprising that fibres over reflector curves have
self-intersection 0, 
whereas fibres over cone points have self-intersection 1.

The calculation of $v_2(M)$ when $\pi=(Z\oplus(Z/2Z))*_ZD$ 
given in Theorem 10.16 of \cite{Hi} is wrong. 
(In fact $Res_\pi^\kappa(S^2)\not=0$, in the notation of \cite{Hi}.) 

\section{$\mathbb{S}^2\times\mathbb{E}^2$-manifolds}

In this section we shall assume that $M$ is a closed 4-manifold with
$\chi(M)=0$ and $\pi_2(M)\cong{Z}$
(equivalently, that $\pi$ is virtually $Z^2$).
In Chapter 10 of \cite{Hi} it is shown that there are between 
21 and 24 possible homotopy types of such 4-manifolds. 
Ten are total spaces of $S^2$-bundles over $T$ or $Kb$, 
four are total spaces of $RP^2$-bundles,
and four are mapping tori of self-homeomorphisms of $RP^3\#{RP^3}$.
These bundle spaces are all $\mathbb{S}^2\times\mathbb{E}^2$-manifolds, 
and their homotopy types are detected by the fundamental groups
and Stiefel-Whitney classes.

The uncertainty relates to the three possible fundamental groups
with finite abelianization.
In each case, the action is unique up to an automorphism of the group.
There is one geometric manifold for each of the groups $D*_ZD$
and $(Z\oplus(Z/2Z))*_ZD$, and two for $Z*_ZD$.
By Theorem 13 there is  one other (non-geometric)
orbifold bundle over $S(2,2,2,2)$ (with group $D*_ZD$),
and these five homotopy types are distinct.
Thus there are in fact 23 homotopy types in all.

If $M$ is an orbifold bundle over a flat base then
it follows from Lemma 2 that either 
\begin{enumerate}
\item $M$ is an $S^2$- or $RP^2$-bundle over $T$ or $Kb$; or

\item $B=\mathbb{A}$ or $\mathbb{M}b$; or

\item $B=S(2,2,2,2)$, $P(2,2)$ or $\mathbb{D}(2,2)$.
\end{enumerate}

There are two $S^2$-orbifold bundles with base 
$S(2,2,2,2)=D(2,2)\cup{D(2,2)}$.
The double of $E(2,2)$ is geometric,
whereas $E(2,2)\cup_\tau{E(2,2)}$ is not.

There is just one $S^2$-orbifold bundle with base $\mathbb{D}(2,2)$.
It has geometric total space.

The orbifold $P(2,2)=D(2,2)\cup{Mb}$ is the quotient of 
the plane $\mathbb{R}^2$ by the group of euclidean isometries 
generated by the glide-reflection $t=(\frac12\mathbf{j},
 \left(\begin{smallmatrix}
-1&0\\
0&1
\end{smallmatrix}\right) ) $ and the rotation
$x=(\frac12(\mathbf{i}+\mathbf{j}),-I)$.
There are two $S^2$-orbifold bundles with base $P(2,2)$.
If we fix identifications of $\partial{Mb}$ with $S^1$ and
$\partial{E(2,2)}$ with $S^2\times{S^1}$ then one has total space 
$M=E(2,2)\cup{S^2}\times{Mb}$ and the other has total space 
$M'=E(2,2)\cup_\tau{S^2}\times{Mb}$.
(The bundles with total space $E(2,2)\cup_{(\tau)}S^2\tilde\times{Mb}$ 
are each equivalent to one of these via the automorphism of the base induced
by reflection of $\mathbb{R}^2$ across the principal diagonal.)

The total spaces of these two $S^2$-orbifold bundles are the
two affinely distinct $\mathbb{S}^2\times\mathbb{E}^2$-manifolds
with fundamental group $Z*_ZD$.
Let $T=(\theta,t)$ and $X=(a,x)$, where $\theta=\pm1\in{S^1}$. 
(Equivalently, $\theta=I_3$ or $R_\pi=diag[-1,-1,1]\in{O(3)}$.)
Then $\{t,x\}$ generates a free, discrete, cocompact isometric action 
of $Z*_ZD$ on $S^2\times{R}^2$.
The subgroup $\kappa\cong{Z}\rtimes_{-1}Z$ is generated by $T$ and $(XT)^2$.
These manifolds are not homotopy equivalent, by Theorem 12.

\section{surgery}

If $\pi$ is a surface group or has a surface group  as an
index-2 subgroup then $Wh(\pi)=0$, by Theorem 6.1 of \cite{Hi}.
Therefore homotopy equivalences of manifolds with such 
fundamental groups are simple.

Let $M$ be a closed 4-manifold with $\pi_2(M)\cong{Z}$ and $\chi=0$.
Then there are nine possibilities for $\pi$.
The relevant surgery obstruction groups
can be computed (or shown to be not finitely generated) in most cases,
via the Shaneson-Wall exact sequences and the results of \cite{CD} 
on $L_n(D,w)$.
L\"uck has settled the one case in which such reductions do not easily apply
\cite{Lu10}.
(The groups $L(\pi)\otimes\mathbb{Z}[\frac12]$ are computed for all 
aspherical 2-orbifold groups $\pi$ when $w$ is trivial in \cite{LS00}.)

If $\pi\cong{Z^2}$ or $Z\rtimes_{-1}Z$ then $M$ is homeomorphic 
to the total space of an $S^2$-bundle over $T$ or $Kb$.
(See Theorem 6.11 of \cite{Hi}.)
If $\pi\cong{Z^2\times{Z/2Z}}$ then $|S_{TOP}(M)|=8$,
while if $\pi\cong{Z\rtimes_{-1}Z}\times{Z/2Z}$
then $8\leq|S_{TOP}(M)|\leq32$.
(See Theorems 6.13 and 6.14 of \cite{Hi}.)
If $\pi\cong{D}\times{Z^-}$ then $L_1(\pi,w)=0$ and $|S_{TOP}(M)|\leq16$.

In each of the remaining cases the structure sets are infinite.
Let $\sigma$ be the automorphism of $D=Z/2Z*Z/2Z$ 
which interchanges the factors.
Let $I_\pi:\pi/\pi'\to{L_1(\pi)}$ be the natural transformation described in
\S6.2 of \cite{Hi}. Then we have

\begin{enumerate}
\item $L_1(D\times{Z})\cong{L_1(D)}\cong{Z^3}$ \cite{CD}.
The direct summand $L_1(Z)\cong{Z}$ is the image of $I_\pi$.
 
\item $L_1(D\rtimes_\sigma{Z})\cong\mathrm{Ker}(1-L_0(\sigma))\cong{Z^2}$. 
The direct summand $L_1(Z)\cong{Z}$ is the image of $I_\pi$.

\item $L_1(D\rtimes_\sigma{Z^-},w)\cong\mathrm{Ker}(1+L_0(\sigma))\cong{Z}$.

\item $D*_ZD$ retracts onto $D(-,-)=Z/2Z^-*Z/2Z^-$, 
compatibly with $w$. Hence $L_1(\pi,w)$ is not finitely generated \cite{CD}.

\item $(Z\oplus(Z/2Z))*_ZD$ retracts onto $D(-,-)=Z/2Z^-*Z/2Z^-$, 
compatibly with $w$. Hence $L_1(\pi,w)$ is not finitely generated \cite{CD}.

\item 
$L_1(Z*_ZD,w)$ has an infinite $UNil$ summand, 
of exponent 4 \cite{Lu10}.
(However $Z*_ZD$ does not surject to $D$.)
\end{enumerate}

In order to estimate the number of homeomorphism types
within each homotopy type we must consider the 
actions of the groups $E(M)$ of homotopy classes of 
self-homotopy equivalences.
(The image of $I_\pi$ acts trivially in $S_{TOP}(M)$,
by Theorem 6.7 of \cite{Hi}.)

Let $M$ be a closed 4-manifold with $\widetilde{M}\simeq{S^2}$.
As observed above, if $M$ is the total space of an orbifold bundle
then $Aut(\pi)$ and $Aut(\pi_2(M))$ act on $M$ via homeomorphisms.
Thus in order to understand the action of $E(M)$ on $S_{TOP}(M)$
it is sufficient to consider the action of the subgroup $K_\pi(M)$ 
of self-homotopy equivalences which induce the identity
on $\pi$ and $\pi_2(M)$.
(Note also that if $f:M\to{M}$ is a self-map such that $\pi_2(f)=id$ then
lifts of $f$ to $\widetilde{M}$ are homotopic to the identity, 
and so $\pi_k(f)=id$ for all $k\geq2$.)

We may assume that $M_o=M\setminus{intD^4}$ is homotopy equivalent 
to a 3-complex.
Fix a basepoint $*\in{M_o}$.
Let $P_3(M)=M\cup{e^{\geq5}}$
be the 3-stage of the Postnikov tower for $M$.
(Thus $\pi_i(M)\cong\pi_i(P_3(M))$ for $i\leq3$ and $\pi_j(P_3(M))=0$ for all $j>3$).
If $(X,*)$ is a based space let $E_*(X)$ be the group 
of based homotopy classes of based self-homotopy equivalences. 
If $f\in{E_*(M)}$ is in the kernel of the natural homomorphism from 
$E_*(M)$ to $E_*(P_3(M))$ then we may assume that $f|_{M_o}$ is the identity,
by cellular approximation.
Thus $f$ differs from $id_M$ by at most a pinch map corresponding to
$\eta{S\eta}\in\pi_4(\widetilde{M})=Z/2Z$.

Let $K_\#$ be the kernel of the natural homomorphism 
from $E_*(P_3(M))$ to $\Pi_{j\leq3}{Aut(\pi_j)}$.
Let $P=P_2(M)$ be the 2-stage of the Postnikov tower for $M$.
Then $K_\#(M)$ maps onto $K_\#$, with kernel of order $\leq2$.
There is an exact sequence
\begin{equation*}
\begin{CD}
H^1(\pi;Z^u)@>\Delta>>{H^3(P;\mathbb{Z})}\to{K_\#}
\to{H^2(\pi;Z^u)}@>\rho>>{H^3(P;\mathbb{Z})},
\end{CD}
\end{equation*}
and the image of ${H^3(P;\mathbb{Z})}$ under the second homomorphism is central.
The homomorphism $\Delta$ involves the second $k$-invariant
$k_2(M)\in{H^4(P;\mathbb{Z})}$ 
and factors through the finite group $H^3(\pi;\mathbb{Z})$.
The kernel of $\rho$ is the isotropy subgroup of $k_2(M)$ 
under the action of $H^2(\pi;Z^u)$ on $P$.
(See Corollary 2.9 of \cite{Ru92}.)

Since $v.c.d.\pi=2$ spectral sequence arguments show that 
$H^i(\pi;Z^u)$ is commensurable 
with $H^0(Z/2Z;H^i(\kappa;\mathbb{Z})\otimes{Z^u})$, for all $i$,
and $H^3(P;\mathbb{Z})$ is commensurable with $H^1(\pi;\mathbb{Z})$.
Thus $K_\#$ is a finitely generated, nilpotent group.
In particular, if $\pi/\pi'$ is finite then $K_\#$ is finite, 
and so there are infinitely many homeomorphism types within 
each such homotopy type.

However, if $\pi\cong{D}\times{Z}$ or $D\rtimes{Z}$ then $K_\#$ is infinite,
and it is not clear how this group acts on $S_{TOP}(M)$.

\section{surface bundles over $RP^2$}

Let $F$ be a closed aspherical surface and 
$p:M\to{RP^2}$ be a bundle with fibre $F$,
and such that $\pi_2(M)\cong{Z}$.
(This condition is automatic if $\chi(F)<0$.) 
Then $\pi=\pi_1(M)$ acts nontrivially on $\pi_2(M)$.
The covering space $M_\kappa$ associated to the kernel $\kappa$ 
of the action is an $F$-bundle over $S^2$,
and so $M_\kappa\cong{S^2}\times{F}$,
since all such bundles are trivial.
The projection admits a section if and only if $\pi\cong\kappa\rtimes{Z/2Z}$. 

The product $RP^2\times{F}$ is easily characterized.

\begin{theorem}
Let $M$ be a closed $4$-manifold with fundamental group $\pi$,
and let $F$ be an aspherical closed surface.
Then the following are equivalent.
\begin{enumerate}
\item$M\simeq{RP^2}\times{F}$;

\item$\pi\cong{Z/2Z}\times\pi_1(F)$, $\chi(M)=\chi(F)$ and $v_2(M)=0$;

\item$\pi\cong{Z/2Z}\times\pi_1(F)$, $\chi(M)=\chi(F)$ and $M\simeq{E}$,
where $E$ is the total space of an $F$-bundle over $RP^2$.
\end{enumerate}
\end{theorem}

\begin{proof}
Clearly $(1)\Rightarrow(2)$ and (3).
If (2) holds then $M$ is homotopy equivalent to the total space of an
$RP^2$-bundle over $F$, by Theorem 5.16 of \cite{Hi}.
This bundle must be trivial since $v_2(M)=0$.
If (3) holds then there are maps $q:M\to{F}$ and $p:M\to{RP^2}$
such that $\pi_1(p)$ and $\pi_1(q)$ are the projections of $\pi$ onto its
factors and $\pi_2(p)$ is surjective.
The map $(p,q):M\to{RP^2}\times{F}$ is then a homotopy equivalence.
\end{proof}

The implication $(3)\Rightarrow(1)$ fails if $F=RP^2$ or $S^2$.

We may assume henceforth that $\pi$ is not a product.
The fixed points of an involution of an orientable surface must be all cone
points (if the involution is orientation-preserving)
or all on reflector curves (if the involution is orientation-reversing).

\begin{theorem}
A closed orientable $4$-manifold $M$ is homotopy equivalent to 
the total space of an $F$-bundle over $RP^2$ with a section
if and only if $\pi=\pi_1(M)$ has an element of order $2$,
$\pi_2(M)\cong{Z}$ and 
$\kappa=\mathrm{Ker}(u)\cong\pi_1(F)$, 
where $u$ is the natural action of $\pi$ on $\pi_2(M)$.
\end{theorem}

\begin{proof}
The conditions are clearly necessary.
If they hold, then $M$ is homotopy equivalent to an $S^2$-orbifold bundle space 
(since it is not homotopy equivalent to an $RP^2$-bundle space).
The base orbifold must have a reflector curve, 
by Lemma 2.
Therefore $M\simeq{M_{st}}$, 
which is the total space of an $F$-bundle over $RP^2$ with a section.
\end{proof}

Orientability is used here mainly to ensure that 
the base orbifold has a reflector curve.

When $\pi$ is torsion-free $M$ is homotopy equivalent to 
the total space of an $S^2$-bundle over a surface $B$, 
with $\pi=\pi_1(B)$ acting nontrivially on the fibre.
Inspection of the geometric models for such bundle spaces 
shows that if also $v_2(M)\not=0$ then the bundle space fibres over $RP^2$. 
(See Theorems 10.8 and 10.9 of \cite{Hi}.)
Is the condition $v_2(M)\not=0$ necessary?

The standard $\mathbb{S}^2\times\mathbb{E}^2$-manifold with group $Z*_ZD$
fibres over $RP^2$, with fibre $Kb$.
Does the other example (constructed using $\theta=-1$)
also fibre over $RP^2$?


\end{document}